\documentclass[a4paper,12pt]{article}
\usepackage{latexsym}
\usepackage{amsmath}
\usepackage{amssymb}
\usepackage{enumerate}

\usepackage{color}

\usepackage{theorem}
\newtheorem{theorem}{Theorem}[section]

\newtheorem{lemma}[theorem]{Lemma}

\newtheorem{MT}{Main Theorem}

\newtheorem{CMT}{Corollary to Main Theorem}

\theorembodyfont{\rmfamily}
\newtheorem{proof}{\textmd{\textit{Proof.}}}

\makeatletter

\@addtoreset{equation}{section}
\makeatother

\newcommand{\qedd}{\hfill \Box}

\newcommand{\wt}{\widetilde}

\def\diam{\mathop{\mathrm{diam}}\nolimits}

\title{Maximal diameter sphere theorem for manifolds with nonconstant radial curvature
\footnote{Mathematics Subject Classification (2010)\,: 53C22.}
.}
\author{Nathaphon BOONNAM}
\date{}
\begin{document}

\maketitle

\begin{abstract}
We generalize the maximal diameter sphere theorem due to Toponogov by means of the radial curvature. As a corollary to our main theorem, we prove that for a complete connected Riemannian $n$-manifold $M$ having  radial sectional curvature at  a point   bounded from below by the radial curvature function of an ellipsoid of prolate type,  the diameter of $M$ does not exceed the diameter of the ellipsoid, and  if the diameter of $M$ equals that of the ellipsoid, then $M$ is isometric to the $n$-dimensional ellipsoid of revolution.
\end{abstract}

\section{Introduction}
The maximal diameter sphere theorem due to Toponogov says that for a complete connected Riemannian manifold $M$ whose sectional curvature is bounded from below by a positive constant $H$, 
the diameter of $M$ does not exceed $\pi/\sqrt{H}$ and if the diameter of $M$ equals $\pi/\sqrt{H}$, then $M$ is isometric to the sphere with radius $\sqrt{H}$. This theorem was generalized by Cheng \cite{Ch} for a complete connected Riemannian manifold whose Ricci curvature is bounded from below by a positive constant $H$.

A natural extension of the maximal diameter sphere theorem by the radial curvature would be:

For a complete connected Riemannian manifold $M$ whose radial sectional curvature at a point $p$ is not less than a positive constant $H$,
\begin{description}
	\item[\textnormal{(A)}] is the diameter of $M$ at most $\pi/\sqrt{H}$?
	\item[\textnormal{(B)}] Furthermore, if the diameter of $M$ equals $\pi/\sqrt{H}$, is $M$ isometric to the sphere with the radius $\sqrt{H}$? 
\end{description}

Notice that the problem (A) can be affirmatively solved. It is an easy consequence from Theorem \ref{th2.2} (or  the Main theorem in \cite {SST}).
Here, we define the radial plane and radial curvature from a point $p$ of a complete connected Riemannian manifold $M$. For each point $q$ distinct from the point $p$, a $2$-dimensional linear subspace $\sigma$ of $T_qM$ is called a {\it radial plane} at $q$ if there exists a unit speed minimal geodesic segment $\gamma:[0,d(p,q)]\to M$ satisfying $\sigma \ni \gamma'(d(p,q))$ and the sectional curvature $K(\sigma)$ of $\sigma$ is called a {\it radial curvature} at $p$.

The problem (B) is still open, but one can generalize the maximal diameter sphere theorem for a manifold which {\it has radial curvature at a point bounded from below by the radial curvature function  of a $2$-sphere of revolution,} which will be defined later,  if the $2$-sphere of revolution belongs to a certain class. 

For introducing this class of a $2$-sphere of revolution, we start to define a $2$-sphere of revolution.

Let $\wt M$ denote a complete Riemannian manifold homeomorphic to a $2$-sphere. 
$\wt M$ is called a {\it $2$-sphere of revolution}  if $\wt M$ admits a point $\tilde p$ such that for any two points $\tilde q_1,\tilde q_2$ on $\wt M$ with $d(\tilde p,\tilde q_1)=d(\tilde p,\tilde q_2)$, where $d(\;,\;)$ denotes the Riemannian distance function, 
there exists an isometry $f$ on $\wt M$ satisfying $f(\tilde q_1)=\tilde q_2$ and $f(\tilde p)=\tilde p$. The point $\tilde p$ is called a {\it pole} of $\wt M$. 
It is proved in \cite{ST} that $\wt M$ has another pole $\tilde q$ and the Riemannian metric  $g$ of $\wt M$ is expressed as $g=dr^2+m(r)^2d\theta^2$ on $\wt M \setminus \{\tilde p,\tilde q\}$, where $(r,\theta)$ denote geodesic polar coordinates around $\tilde p$ and
$$m(r(x)):=\sqrt{g\bigg(\bigg(\frac{\partial}{\partial \theta}\bigg)_x,\bigg(\frac{\partial}{\partial \theta}\bigg)_x\bigg)}.$$
Hence $\wt M$ has a pair of poles $\tilde p$ and $\tilde q$. In what follows, $\tilde p$ denotes a pole of $\wt M$ and we fix it. Each unit speed geodesic emanating from $\tilde p$ is called a {\it meridian}. It is observed in \cite{ST} that each meridian $\mu:[0,4a] \to \wt M$, where $a:=\frac{1}{2}d(\tilde p,\tilde q)$, passes through $\tilde q$  and is  periodic, hence, $\mu(0)=\mu(4a)=\tilde p, \mu'(0)=\mu'(4a)$.
The function $G\circ \mu : [0,2a]\to R$ is called the {\it radial curvature function} of $\wt M,$ where $G$ denotes the Gaussian curvature of $\wt M$.

A $2$-sphere of revolution $\wt M$ with a pair of poles $\tilde p$ and $\tilde q$ is called a {\it model surface} if $\wt M$ satisfies the following two properties:
\begin{description}
	\item[{\normalfont (1.1)}] $\wt M$ has a reflective symmetry with respect to the {\it equator}, $r=a=\frac{1}{2}d(\tilde p,\tilde q)$.
	\item[{\normalfont (1.2)}] The Gaussian curvature $G$ of $\wt M$ is strictly decreasing along a meridian from the point $\tilde p$ to the point on the equator.
\end{description}
A typical example of a model surface is an ellipsoid of prolate type, i.e., the surface defined by
$$\frac{x^2+y^2}{a^2}+\frac{z^2}{b^2}=1, \quad b>a>0.$$
The points $(0,0,\pm b)$ are   a pair of poles and $z=0$ is the equator.

The Gaussian curvature of  a model surface  is not always  positive everywhere.
In \cite {ST}, an interesting model surface was introduced. 
The surface generated by the $(x,z)$-plane curve $(m(t),0,z(t))$
is a model surface, where
$$m(t):=\frac{\sqrt{3}}{10}\bigg(9\sin \frac{\sqrt{3}}{9}t+7\sin \frac{\sqrt{3}}{3}t\bigg), \quad z(t):=\int^t_0 \sqrt{1-m'(t)^2}dt.$$
It is easy to see  that  the Gaussian curvature 
of the equator $r=3\sqrt{3}\pi/2$ is $-1.$

Let $M$ be a complete connected  $n$-dimensional Riemannian manifold with a base point $p$. $M$ is said to have {\it radial sectional curvature at p bounded from below by the radial curvature function $G\circ \mu$ of  a model surface} $\wt M$ if for any point $q(\ne p)$ and any radial plane  $\sigma\subset T_q M$ at $q$, the sectional curvature $K(\sigma)$  of $M$ satisfies 
$K(\sigma)\geq G\circ\mu(d(p,q))$.

For each $2$-dimensional model $\wt M$ with a Riemannian metric $dr^2+m(r)^2d\theta^2$, we define an $n$-dimensional model $\wt M^n$ homeomorphic to an $n$-sphere $S^n$ with a Riemannian metric
$$g^*=dr^2+m(r)^2d\Theta^2,$$
where $d\Theta^2$ denotes the Riemannian metric of the $(n-1)$-dimensional unit sphere $S^{n-1}(1)$.
For example, the $n$-dimensional model of the ellipsoid above  is the $n$-dimensional ellipsoid defined by
$$\sum_{i=1}^n \frac{x_i^2}{a^2}+\frac{x_{n+1}^2}{b^2}=1.$$

In this paper, we generalize the maximal diameter sphere theorem as follows:
\begin{MT}
Let $M$ be a complete connected  $n$-dimensional Riemannian manifold with a base point $p$. If $M$ has radial sectional curvature at $p$ bounded from below by the radial curvature function of a model surface  $\wt M$, then the diameter of $M$ does not exceed the diameter of $\wt M$ and furthermore  if the diameter of $M$ equals that of $\wt M,$ then $M$ is isometric to the $n$-dimensional model $\wt M^n$.
\end{MT}

As a corollary, we get an interesting result:

\begin{CMT}
For any complete connected n-dimensional Riemannian manifold $M$ having radial sectional curvature at a point $p$ bounded  from below by the radial curvature function of the ellipsoid $\wt M$ defined by
$$\frac{x^2+y^2}{a^2}+\frac{z^2}{b^2}=1, \quad b>a>0,$$
the diameter of $M$ does not exceed the diameter of  $\wt M$
and if the diameter of $M$ equals that of $\wt M$, then $M$ is isometric to the $n$-dimensional ellipsoid $\sum_{i=1}^n \frac{x_i^2}{a^2}+\frac{x_{n+1}^2}{b^2}=1$.
\end{CMT}
We refer to [CE] for basic tools in Riemannian Geometry, and [SST] for some properties of geodesics on a surface of revolution.

The present author would like to deeply express thanks to Professor Minoru Tanaka for suggesting the Main Theorem  and giving him various comments.

\section{Preliminaries}

We need the following two theorems, which was proved by Sinclair and Tanaka \cite{ST}.

\begin{theorem}\label{th2.1}
{\normalfont (\cite{ST})}
Let $M$ be a $2$-sphere of revolution with a pair of poles $p,q$ satisfying the following two properties,
\begin{description}
	\item[\textnormal{(i)}] $M$ is symmetric with respect to the reflection fixing $r=a$, where $2a$ denotes the distance between $p$ and $q$.
	\item[\textnormal{(ii)}] The Gaussian curvature $G$ of $M$ is monotone along a meridian from the point $p$ to the point on $r=a$.
\end{description}
Then the cut locus of a point $x \in M \setminus{\{p,q\}}$ with $\theta (x)=0$ is a single point or a subarc  of the opposite half meridian $\theta=\pi$ (resp. the parallel $r=2a-r(x)$) when $G$ is decreasing (resp. increasing) along a meridian from $p$ to the point on $r=a$. Furthermore, if the cut locus of a point $x \in M \setminus{\{p,q\}}$ is a single point, then the Gaussian curvature is constant.
\end{theorem}

Here, we review the notion of a cut point and a cut locus. Let $M$ be a complete Riemannian manifold with a base point $p$. Let $\gamma:[0,a] \to M$ denote a unit speed minimal geodesic segment emanating from $p=\gamma(0)$ on $M$. If any extended geodesic segment $\bar \gamma:[0,b] \to M$ of $\gamma$, where $b>a$, is not minimizing arc joining $p$ to $\bar \gamma (b)$ anymore, then the endpoint $\gamma (a)$ of the geodesic segment is called a {\it cut point} of $p$ along $\gamma$. For each point $p$ on $M$, the {\it cut locus} $C_p$ is defined by the set of all cut points along the minimal geodesic segments emanating from $p$. It is known (for  example see \cite {SST})  that the cut locus has a local tree structure for  2-dimensional Riemannian manifolds.

\begin{theorem}\label{th2.2} 
Let $M$ be a complete connected  $n$-dimensional Riemannian manifold with a base point $p$ such that $M$ has  radial sectional curvature at $p$ bounded from below by the radial curvature function of a 2-sphere of revolution $\wt M$ with a pair of poles $\tilde p,\tilde q.$  Suppose that the cut locus of any point on $\wt M$ distinct from its two poles is a subset of the half meridian opposite to the point. Then for each geodesic triangle $\triangle (pxy)$ in $M$, there exists a geodesic triangle $\wt \triangle (pxy):=\triangle (\tilde p \tilde x \tilde y)$ in $\wt M$ 
such that
\begin{equation}\label{eq2.1}
d(p,x)=d(\tilde p, \tilde x),\quad 
d(p,y)=d(\tilde p, \tilde y),\quad 
d(x,y)=d(\tilde x, \tilde y),
\end{equation}
and such that
\begin{equation}\label{eq2.2}
\angle(pxy) \geqslant \angle(\tilde p \tilde x \tilde y),\quad 
\angle(pyx) \geqslant \angle(\tilde p \tilde y \tilde x),\quad 
\angle(xpy) \geqslant \angle(\tilde x \tilde p \tilde y).
\end{equation}
Here, $\angle (pxy)$ denotes the angle at the vertex $x$ of the geodesic triangle $\triangle (pxy)$.
\end{theorem}


\section{Proof of Main Theorem}
Let $M$ be a complete connected  $n$-dimensional  Riemannian manifold with a base point $p$ and $\wt M$  a $2$-sphere of revolution with a pair of poles $\tilde p,\tilde q$ satisfying (1.1) and (1.2) in the introduction, i.e., a model surface.

From now on, we assume that $M$ has radial sectional curvature at $p$ bounded from below by the radial curvature function of $\wt M$. By scaling the Riemannian metrics of $M$ and $\wt M$, we may assume that $2a=\pi$. 

\begin{lemma}\label{lem3.1}
The perimeter of any geodesic triangle $\wt \triangle (pxy)$ of $\wt M$ does not exceed $2\pi$, i.e.,
\begin{equation}\label{eq3.1}
d(\tilde p,\tilde x)+d(\tilde p,\tilde y)+d(\tilde x,\tilde y) \leqslant 2\pi.
\end{equation}
\end{lemma}

\begin{proof}
Since $d(\tilde p,\tilde q)=2a=\pi$, it follows from the triangle inequality that
\begin{align*}
d(\tilde x,\tilde y)	&\leqslant	d(\tilde q,\tilde x)+d(\tilde q,\tilde y) \\
				&=		(\pi -d(\tilde p,\tilde x))+(\pi -d(\tilde p,\tilde y)) \\
				&=		2\pi -d(\tilde p,\tilde x)-d(\tilde p,\tilde y).
\end{align*}
Therefore, the inequality \eqref{eq3.1} holds.
$\qedd$
\end{proof}

\begin{lemma}\label{lem3.2}
The perimeter of a geodesic triangle $\triangle (pxy)$ of $M$ does not exceed $2\pi$.
\end{lemma}

\begin{proof}
Let $\triangle (pxy)$ be any geodesic triangle of $M$. From Theorem \ref{th2.2}, we get a geodesic triangle $\wt \triangle (pxy)$ of $\wt M$ satisfying \eqref{eq2.1}. Hence, by Lemma \ref{lem3.1}, the perimeter of $\triangle(pxy)$ does not exceed $2\pi$.
$\qedd$
\end{proof}

\begin{lemma}\label{lem3.3}
The diameter of $\wt M$ equals $\pi,$ where the diameter $\diam \wt M$  of $\wt M$ is defined by
$$\diam \wt M:=\max\{d(\tilde x,\tilde y)|\tilde x, \tilde y \in \wt M\}.$$
\end{lemma}

\begin{proof}
Choose any points $\tilde x, \tilde y$ on $\wt M$. By the triangle inequality,
\begin{equation}\label{eq3.2}
d(\tilde x,\tilde y) \leqslant d(\tilde p,\tilde x)+d(\tilde p,\tilde y).
\end{equation}
Thus, by combining \eqref{eq3.1} and \eqref{eq3.2}, we obtain
$$d(\tilde x,\tilde y) \leqslant \pi=d(\tilde p,\tilde q)$$
for any $\tilde x,\tilde y$ on $\wt M$.
$\qedd$
\end{proof}

\begin{lemma}\label{lem3.4}
The diameter $\diam  M$ of $M$ does not exceed the diameter of $\wt M$.
\end{lemma}

\begin{proof}
Choose a pair of points $x,y \in M$ satisfying $d(x,y)=\diam M$.
Suppose that $x=p$ or $y=p.$
By the Rauch comparison theorem, there does not exist a minimal geodesic segment emanating from $p$ whose length exceeds $\pi,$ since the manifold $M$
has radial curvature at p bounded from below by the radial curvature function of the model surface $\wt M.$ Thus, $\diam M=d(x,y)\leqslant\pi.$ 
Suppose that $x\ne p$ and $y\ne p.$ Then,
for the geodesic triangle $\triangle (pxy)$ in $M$, there exists a triangle geodesic $\wt \triangle (pxy)$ in $\wt M$ satisfying \eqref{eq2.1}. Hence, we obtain $\diam M = d(\tilde x,\tilde y) \leqslant \diam \wt M$.
$\qedd$
\end{proof}

\begin{lemma}\label{lem3.5}
If $\diam M = \diam \wt M$, then there exists a point $q\in M$ with $d(p,q)=\diam \wt M.$
\end{lemma}

\begin{proof}
Let $x,y \in M$ be points satisfying  $\pi=\diam M = d(x,y)$. Supposing  that $x\ne p$ and $y\ne p,$ we will get a contradiction. Then,  there exists a geodesic triangle $\triangle(pxy)$ with $d(x,y)=\pi.$ 

 It follows from  Theorem \ref{th2.2} that there exists a geodesic triangle $\wt \triangle(pxy)$ corresponding to $\triangle(pxy)$ satisfying $d(\tilde x,\tilde y)=d(x,y)=\pi$. By the triangle inequality, $d(\tilde p,\tilde x)+d(\tilde p,\tilde y) \geqslant d(\tilde x,\tilde y)=\pi$, and Lemma \ref{lem3.1}, we get 
$$d(\tilde p,\tilde x)+d(\tilde p,\tilde y)=\pi=d(\tilde x,\tilde y).$$ 
This means that the subarc $\alpha$ (passing through $\tilde p$) of the meridian joining $\tilde x$ to $\tilde y$ is minimal. 
Hence the complementary  subarc of $\alpha$ in the meridian is also a  minimal geodesic segment joining  $\tilde x$ to $\tilde y,$ since the length of each meridian is $2\pi$. Therefore,
by Theorem \ref{th2.1} $\tilde y$ is a unique cut point of $\tilde x$ and
hence,  the Gaussian curvature $G$ of $\wt M$ is constant. We get a contradiction since $G$ is strictly decreasing along a meridian from $p$ to the point on the equator. This implies the existence of the point $q.$ 
$\qedd$
\end{proof}


\begin{lemma}\label{lem3.6}
If there exists a point $q\in M$ with $d(p,q)=\diam \wt M,$ then $q$ is a unique cut point of $p,$ and
$$K(\sigma)=G\circ \mu(d(p,x))$$
holds for any point $x \in M \setminus \{p\}$ and any radial plane $\sigma$ at $x$.
\end{lemma}

\begin{proof}
Since the point  $q$ is a farthest point from $p,$ 
$q$ is  a cut point of $p.$ 
Choose any point $x\in M \setminus \{p,q\}.$
By the triangle inequality,
$$d(p,x)+d(x,q)\geqslant d(p,q)=\pi$$
and by Lemma \ref{lem3.2},
$$d(p,x)+d(x,q)+d(p,q)\leqslant 2\pi.$$
Hence, we get
$$d(p,x)+d(x,q)=d(p,q)=\pi$$
and it is easy to see that $q$ is  a unique cut point of $p.$

Next, we will prove that $K(\sigma)=G\circ \mu(d(p,x))$ for any $x\in M\setminus\{p,q\}$ and any radial plane $\sigma$ at $x.$
Suppose that there exist  a point  $x \in M \setminus \{p,q\}$ and a radial plane $\sigma$ at $x$ such that $K(\sigma)>G\circ \mu(d(p,x))$. 
Let $\gamma:[0,\pi] \to M$ denote the minimal geodesic segment emanating from $p$ passing through $x$. 
Choose a unit tangent vector $v \in\sigma\subset T_xM $ orthogonal to $\gamma'(d(p,x)).$
Let  $Y(t)$ denote the Jacobi field along $\gamma(t)$ satisfying $Y(0)=0$ and  $Y(d(p,x))=v,$ and hence
 $\sigma$ is spanned by $Y(d(p,x))$ and $\gamma'(d(p,x)).$
By the  Rauch comparison theorem, there exists a conjugate point $\gamma(t_1)$ of $p$ along $\gamma$ for some $t_1\in(0,\pi), $ since $K(\sigma)>G\circ\mu(d(p,x))$ and the sectional  curvature of the radial plane spanned by $Y(t)$ and $\gamma'(t)$ is not less than $G\circ\mu(t)$  for each $t\in(0,\pi).$
This contradicts the fact that the geodesic segment $\gamma$  is minimal.
$\qedd$
\end{proof}

{\it Proof of Main Theorem.} 
The first claim is clear from Lemma \ref{lem3.4}. Suppose that $\diam M=\diam \wt M.$  By Lemmas \ref{lem3.5} and \ref{lem3.6}, $K(\sigma)=G\circ \mu(d(p,x))$ for any point $x \in M \setminus \{p\}$ and any radial plane $\sigma$ at $x$.
Thus, it follows from Lemma 1 and Theorem 3 in \cite{KK}  that $M$ is isometric to the $n$-dimensional model of $\wt M.$
Incidentally, the explicit isometry $\varphi$ between $M$ and the $n$-dimensional model of $\wt M$ is given by

\[\varphi(x):=
	\begin{cases}
		\exp_{\tilde p}\circ I\circ \exp_p^{-1}(x) &\text{if}~~~x\neq q\\
		\tilde q     &\text{if}~~~x=q,
	\end{cases}
\]
where $I:T_pM \to T_{\tilde p}\wt M$ denotes a linear isometry and $q$ denotes the unique cut  point of  $p.$


\bigskip

\begin{center}
Nathaphon BOONNAM \\

\medskip
Department of Mathematics\\
Tokai University\\
Hiratsuka City, Kanagawa\\
259\,--\,1292 Japan

\medskip
{\it e--mail } :
{\tt nut4297nb@gmail.com}\\

\end{center}

\end{document}